\documentclass[12pt]{amsart}
\usepackage{tikz}
\usepackage{amsmath,amssymb}
\usepackage{amsopn}
\usepackage{mathrsfs}
\usepackage{bm}
\usepackage{url}
\usepackage{listings,jvlisting}
\usepackage{tcolorbox}
\usepackage{ascmac}
\usepackage{amsthm}
\usepackage{ulem}
\usepackage{graphics}
\usepackage[export]{adjustbox}
\usepackage{tikz}
\usepackage{subcaption}
\usepackage[margin=25truemm,bottom=30truemm,top=30truemm]{geometry}
\usepackage{comment}
\usepackage{hyperref}
\usepackage{xcolor}
\usepackage{mathdots}
\usepackage{tikz-3dplot}
\usepackage[title]{appendix}

\makeatletter
\def\subsection{\@startsection{subsection}{2}%
  \z@{.5\linespacing\@plus.7\linespacing}{.5\linespacing}%
  {\normalfont\bfseries}}
\makeatother

\usetikzlibrary{knots}
\usetikzlibrary{braids}
\usetikzlibrary{decorations.markings}
\usetikzlibrary{hobby}
\usetikzlibrary{arrows.meta}
\usetikzlibrary{positioning}
\usetikzlibrary{calc}

\theoremstyle{definition}
\newtheorem{definition}{Definition}[section]

\theoremstyle{plain}
\newtheorem{theorem}[definition]{Theorem}
\newtheorem{corollary}[definition]{Corollary}
\newtheorem{proposition}[definition]{Proposition}
\newtheorem{example}[definition]{Example}
\newtheorem{lemma}[definition]{Lemma}

\theoremstyle{remark}
\newtheorem{remark}[definition]{Remark}

\numberwithin{equation}{section}

\newcommand{\CC}{\mathbb{C}}

\newcommand{\ZZ}{\mathbb{Z}}
\newcommand{\QQ}{\mathbb{Q}}

\newcommand{\ie}{\emph{i.e.}\@ifnextchar.{\!\@gobble}{}}
\newcommand{\eg}{\emph{e.g.}\@ifnextchar.{\!\@gobble}{}}
\newcommand{\etc}{etc\@ifnextchar.{}{.\@}}

\newcommand{\QC}{\tilde{Q}}

\newcommand{\crossing}{
    \begin{tikzpicture}[scale=0.5,baseline={(0,-0.1)}]
        \draw (-1,1) to (1,-1);
        \draw[white, line width=7pt] (0.5,0.5) to (-0.5,-0.5);
        \draw (1,1) to (-1,-1);
    \end{tikzpicture}
}

\newcommand{\crossingflip}{
    \begin{tikzpicture}[scale=0.5,baseline={(0,-0.1)},rotate=90,transform shape]
        \draw (-1,1) to (1,-1);
        \draw[white, line width=7pt] (0.5,0.5) to (-0.5,-0.5);
        \draw (1,1) to (-1,-1);
    \end{tikzpicture}
}

\newcommand{\crossingzero}{
    \begin{tikzpicture}[scale=0.5,baseline={(0,-0.1)}]
        \draw (1,1) to[curve through={(0.5,0)}] (1,-1);
        \draw (-1,1) to[curve through={(-0.5,0)}] (-1,-1);
    \end{tikzpicture}
}

\newcommand{\crossinginfty}{
    \begin{tikzpicture}[scale=0.5,baseline={(0,-0.1)}]
        \draw (1,1) to[curve through={(0,0.5)}] (-1,1);
        \draw (1,-1) to[curve through={(0,-0.5)}] (-1,-1);
    \end{tikzpicture}
}

\newcommand{\twobridge}{
    \begin{tikzpicture}
        \pic[
            braid/.cd,
            number of strands=3,
            strand 1/.style={black},
            strand 2/.style={black},
            strand 3/.style={black},
            gap=0.1,
            scale=1
        ] {braid={s_2 s_1^{-1} s_2 s_1^{-1}}};
        \node at (1,-5) {\Large$\vdots$};
        \begin{scope}[yshift=-6cm]
            \pic[
                braid/.cd,
                number of strands=3,
                strand 1/.style={black},
                strand 2/.style={black},
                strand 3/.style={black},
                gap=0.1,
                scale=1
            ] {braid={s_2 s_1^{-1}}};
        \end{scope}

        \draw (0,0) to[in=90,out=90] (1,0);
        \draw (1,-8.5) to[in=-90,out=-90] (2,-8.5);
        \draw (2,0) to[in=0,out=90] (0.5,1);
        \draw (0.5,1) to[in=90,out=180] (-1,-4.25);
        \draw (-1,-4.25) to[in=180,out=-90] (-0.25,-9);
        \draw (-0.25,-9) to[in=-90,out=0] (0,-8.5);
    \end{tikzpicture}
}

\begin{document}

\title{On the Mahler measure and root distribution of the $Q$-polynomial of links}
\author{Kotaro Shoji}
\address{Department of Mathematics, Graduate School of Science, Osaka Metropolitan University, 
3-3-138, Sugimoto, Sumiyoshi-ku, Osaka, 558-8585, Japan}
\email{sj26992g@st.omu.ac.jp}
\date{}

\begin{abstract}
    We study the roots and the Mahler measure of the $Q$-polynomial of links.
    We first consider links obtained by adding twists to a pair of parallel
    strands. We prove that the Mahler measure of the transformed
    $Q$-polynomial converges as the number of twists increases. We also show
    that all but a uniformly bounded number of distinct roots of the
    $Q$-polynomial approach the real interval $[-2,2]$. This behavior is
    different from that of the roots of
    the Jones polynomial under twisting. Finally, we compare real and unit-circle
    roots of the Alexander, Jones, and $Q$-polynomials, and give an infinite family of
    $2$-bridge links whose $Q$-polynomials have only real nonzero roots.
\end{abstract}


\maketitle

\section{Introduction}

A knot is an embedding of a circle in three-dimensional space, and a link is
an embedding of a disjoint union of circles. Polynomial invariants are useful
tools for studying knots and links. Examples include the Alexander, Jones,
HOMFLY, Kauffman, and $Q$-polynomials. The roots of these polynomials have also
been studied in many ways.

One of the broader goals of this study is to understand how the roots of knot
polynomials are related to the topology of knots and links and to the geometry
of their complements. In this paper, we consider a first step toward this goal
for the $Q$-polynomial. We study how its roots and Mahler measure change when
twists are added to a link diagram. Twisting is closely related to Dehn surgery
on link complements, so twist families may be useful for studying such a
relation. Although we do not obtain a direct geometric interpretation in this
paper, we hope that our results will be useful for further study of this
question.

The $Q$-polynomial may be viewed as a twin of the Jones polynomial in the
following sense: both are obtained by specializing the Kauffman polynomial
$F_L(a,x)$. To define it, first consider the regular isotopy invariant
$\Lambda_D(a,x)$ of an unoriented link diagram $D$, normalized by
$\Lambda_U(a,x)=1$ for the crossing-free unknot diagram and satisfying
\begin{align*}
    \Lambda_{D_+}(a,x)+\Lambda_{D_-}(a,x)
    &=x\bigl(\Lambda_{D_0}(a,x)+\Lambda_{D_\infty}(a,x)\bigr), \\
    \Lambda_{D^+}(a,x)&=a\Lambda_D(a,x), \\
    \Lambda_{D^-}(a,x)&=a^{-1}\Lambda_D(a,x),
\end{align*}
where $D^+$ and $D^-$ are obtained by adding a positive and a negative
curl, respectively. For an oriented diagram $D$ of a link $L$, the
Kauffman polynomial is
\[
    F_L(a,x)=a^{-w(D)}\Lambda_D(a,x),
\]
where $w(D)$ is the writhe. This normalization makes $F_L$ invariant under
all Reidemeister moves.

\begin{figure}
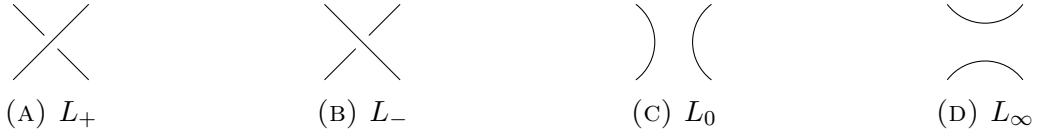

    \begin{subfigure}{0.24\textwidth}
        \centering
        \crossing
        \caption{$L_+$}
    \end{subfigure}
    \begin{subfigure}{0.24\textwidth}
        \centering
        \crossingflip
        \caption{$L_-$}
    \end{subfigure}
    \begin{subfigure}{0.24\textwidth}
        \centering
        \crossingzero
        \caption{$L_0$}
    \end{subfigure}
    \begin{subfigure}{0.24\textwidth}
        \centering
        \crossinginfty
        \caption{$L_\infty$}
    \end{subfigure}
    \caption{The four link diagrams in the skein relation.}
    \label{nanka}
\end{figure}

The Jones polynomial $J_L(q)$ is obtained by setting
$a=-q^{-3/4}$ and $x=q^{1/4}+q^{-1/4}$, with an appropriate normalization,
whereas the $Q$-polynomial is obtained by setting $a=1$ for unoriented links.
This relationship is illustrated in Figure~\ref{fig:twins}.

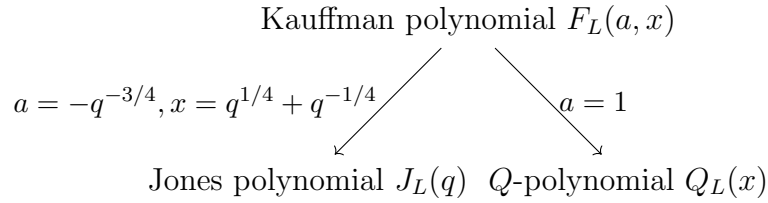
\begin{figure}[htbp]
    \centering
    \begin{tikzpicture}[node distance=2cm, auto]
        \node (K) {Kauffman polynomial $F_L(a, x)$};
        \node (J) [below left of=K, node distance=3cm] {Jones polynomial $J_L(q)$};
        \node (Q) [below right of=K, node distance=3cm] {$Q$-polynomial $Q_L(x)$};
        
        \draw[->] (K) -- node[left, font=\small] {$a=-q^{-3/4}, x=q^{1/4}+q^{-1/4}$} (J);
        \draw[->] (K) -- node[right, font=\small] {$a=1$} (Q);
    \end{tikzpicture}
    \caption{The relationship between the Kauffman, Jones, and $Q$-polynomials.}
    \label{fig:twins}
\end{figure}

Our starting question is how the roots change when the number of twists in a
link diagram increases. We also ask whether the Mahler measure of the
corresponding polynomial has a meaningful limit. For the Jones polynomial,
Champanerkar and Kofman proved a convergence result for the Mahler measure
under twisting~\cite{10.2140/agt.2005.5.1}.

We prove that, as the number of half twists tends to infinity, all but a
bounded number of the distinct roots of the $Q$-polynomial approach the
interval $[-2,2]$. This is in contrast with the behavior of the roots of the
Jones polynomial, for which the unit circle is the natural limiting set. We also
study the Mahler measure after the change of variables
\[
    \QC_L(z)=Q_L(z+z^{-1}).
\]
We prove a convergence result for $M(\QC_L)$ under twisting that is similar
to the known result for the Jones polynomial. Silver and Williams
\cite{MR2050045} discuss possible connections between the Mahler measure of
Alexander polynomials and hyperbolic volume. For the transformed $Q$-polynomial,
we do not obtain a direct relation with hyperbolic volume; the similarity
discussed here is limited to convergence under twisting.

It is also natural to ask whether all roots of a knot polynomial are real or
lie on the unit circle. We therefore compare these properties for the
Alexander, Jones, and $Q$-polynomials, and give an infinite family of
$2$-bridge links whose $Q$-polynomials have only real nonzero roots.

The paper is organized as follows. In Section~2, we recall the definitions of
the $Q$-polynomial, the Mahler measure, and the Jones polynomial. In Section~3,
we prove the results on twisting. In Section~4, we compare real and unit-circle
roots of the Alexander, Jones, and $Q$-polynomials.

\section{Preliminaries}

\subsection{The $Q$-polynomial}

The $Q$-polynomial $Q_L(x)\in\ZZ[x^{\pm1}]$ of an unoriented link $L$
is determined by the skein relation
\[
    Q_{L_+}(x)+Q_{L_-}(x)
    =x\bigl(Q_{L_0}(x)+Q_{L_\infty}(x)\bigr)
\]
and the normalization $Q_U(x)=1$ for the unknot $U$
\cite{MR837528,Ho1985}. Here $L_+,L_-,L_0,L_\infty$ are identical outside
the portions shown in Figure~\ref{nanka}.

For an $s$-component link $L$, we will use the special values
\[
    Q_L(1)=1,
    \qquad
    Q_L(-2)=(-2)^{s-1}.
\]
The $Q$-polynomial is multiplicative under connected sum:
\[
    Q_{K_1\mathbin{\#}K_2}(x)=Q_{K_1}(x)Q_{K_2}(x).
\]

It is convenient to introduce the reciprocal rational function
\[
    \QC_L(z):=Q_L\left(z+z^{-1}\right).
\]
When $L$ is a knot, $Q_L$ has no negative powers, so $\QC_L$ is a Laurent
polynomial. For a general link, $\QC_L$ may have poles at $z=\pm i$.
The substitution $x=z+z^{-1}$ converts the interval $[-2,2]$ into the
unit circle: a number $x$ belongs to $[-2,2]$ if and only if the solutions
of $z+z^{-1}=x$ lie on the unit circle. Thus, questions about roots of
$Q_L$ near $[-2,2]$ can be translated into questions about roots of
$\QC_L$ near the unit circle.

\subsection{Mahler measure}

For a nonzero Laurent polynomial
$f\in\CC[x_1^{\pm1},\ldots,x_s^{\pm1}]$, its Mahler measure is
\[
    M(f)=\exp\left(
    \int_0^1\!\cdots\!\int_0^1
    \log\left|f(e^{2\pi i\theta_1},\ldots,e^{2\pi i\theta_s})\right|
    \,d\theta_1\cdots d\theta_s\right).
\]
For a one-variable polynomial
$f(x)=a\prod_{j=1}^d(x-\gamma_j)$, Jensen's formula
in~\cite{10.1007/BF02417878} gives
\[
    M(f)=|a|\prod_{j=1}^d\max\{1,|\gamma_j|\}.
\]
We extend the Mahler measure multiplicatively to nonzero rational functions
by setting $M(f/g)=M(f)/M(g)$. The Boyd--Lawton theorem, recalled below,
will be the main tool for passing from one-variable specializations to
multivariable Laurent polynomials.

\subsection{The Jones polynomial}

We use the normalization of the Jones polynomial $J_L(q)$ determined by
$J_U(q)=1$ and
\[
    q^{-1}J_{L_+}(q)-qJ_{L_-}(q)
    =\left(q^{1/2}-q^{-1/2}\right)J_{L_0}(q)
\]
for an oriented skein triple $(L_+,L_-,L_0)$. For a knot $K$, this
normalization satisfies the special-value identities
\cite[Section~12]{Jones1987HeckeAR}
\[
    |J_K(1)|=|J_K(i)|
    =\left|J_K\left(e^{2\pi i/3}\right)\right|=1.
\]
These three values will be used in our discussion of root stability.

\section{Mahler measure convergence and roots of the $Q$-polynomial}

For a vector $\mathbf{r}=(r_1,\ldots,r_s)\in\ZZ^s$, define
\begin{equation*}
    \nu(\mathbf{r})
    = \min\left\{\max_{1\leq j\leq s}|w_j|\mathrel{\Big|}
    \mathbf{w}=(w_1,\ldots,w_s)\in\ZZ^s\setminus\{\mathbf{0}\},
    \ \mathbf{r}\cdot\mathbf{w}=0\right\}.
\end{equation*}

\begin{lemma}[Boyd--Lawton \cite{LAWTON1983356}]\label{boydlawton}
    For any nonzero Laurent polynomial
    $f\in\CC[x_1^{\pm1},\ldots,x_s^{\pm1}]$, we have
    \begin{equation*}
        \lim_{\nu(\mathbf{r})\to\infty}
        M\bigl(f(x^{r_1},\ldots,x^{r_s})\bigr)=M(f).
    \end{equation*}
    In particular,
    \begin{equation*}
        \lim_{d\to\infty}
        M\bigl(f(x,x^d,x^{d^2},\ldots,x^{d^{s-1}})\bigr)=M(f).
    \end{equation*}
\end{lemma}

The same conclusion holds for a nonzero rational function: write it as a
quotient of two Laurent polynomials and apply Lemma~\ref{boydlawton} to the
numerator and the denominator.

\begin{theorem}\label{mainthone}
    Let $L_m$ be a link obtained by inserting $m$ half twists into a pair of
    parallel strands in a diagram of an $s$-component link $L$. Then there is
    a polynomial $f\in\ZZ[x^{\pm 1},t^{\pm 1}]$ such that
    \begin{equation}
        \lim_{m\to \infty}M(\QC_{L_m}(x))=M(f(x,t)).
    \end{equation}
\end{theorem}

\begin{proof}
    Applying the skein relation successively at the last crossing in the twist
    region gives the recurrence relation
    \begin{equation}
        \QC_{L_m}(x) - (x+1+\frac{1}{x})\QC_{L_{m-1}}(x) + (x+1+\frac{1}{x})\QC_{L_{m-2}}(x) - \QC_{L_{m-3}}(x) = 0
    \end{equation}
    Its characteristic polynomial factors as
    \[
        \lambda^3-\left(x+1+x^{-1}\right)\lambda^2
        +\left(x+1+x^{-1}\right)\lambda-1
        =(\lambda-1)(\lambda-x)(\lambda-x^{-1}).
    \]
    Therefore, we can express $\QC_{L_m}(x)$ as
    \begin{equation*}
        \QC_{L_m}(x) = C_0 + C_1 x^m +C_2 x^{-m},
    \end{equation*}
    where $C_i\in\QQ(x)$.
    Set
    \[
        F(x,t):=C_0+C_1t+C_2t^{-1}.
    \]
    Then $\QC_{L_m}(x)=F(x,x^m)$, and the rational-function version of
    Lemma~\ref{boydlawton} gives
    \[
        \lim_{m\to\infty}M(\QC_{L_m}(x))=M(F(x,t)).
    \]
    Let $D(x)=(x^2+1)^s(x+1)(x-1)^2$ and set $f(x,t)=D(x)F(x,t)$.
    By Lemma~\ref{mainlemma}, $f\in\ZZ[x^{\pm1},t^{\pm1}]$; moreover,
    $M(F)=M(f)$ since $M(D)=1$. This proves the assertion.
\end{proof}

\begin{corollary}\label{maincor}
    Let $\mathbf{m} =(1, m_1,\ldots,m_\ell)\in\ZZ^{\ell+1}$.
    Let $L_\mathbf{m}$ be a link constructed by inserting
    $m_1,\ldots,m_\ell$ half twists into $\ell$ pairs of parallel strands
    in a diagram of a link $L$. Then there is a polynomial
    $f\in\ZZ[x^{\pm 1},t_1^{\pm 1},\ldots,t_\ell^{\pm 1}]$ such that
    \begin{equation}
        \lim_{\nu(\mathbf{m})\to \infty}M(\QC_{L_\mathbf{m}}(x))
        =M(f(x,t_1,\ldots,t_\ell)).
    \end{equation}
\end{corollary}

\begin{proof}
    We apply the argument successively to the $\ell$ twist sites. At each site,
    the recurrence used in the proof of Theorem~\ref{mainthone} expresses the
    dependence on the corresponding twist parameter as a linear combination of
    $1,t_j,t_j^{-1}$.
    Induction on the number of twist sites therefore gives
    \[
        F(x,t_1,\ldots,t_\ell)
        \in\QQ(x)[t_1^{\pm1},\ldots,t_\ell^{\pm1}]
    \]
    such that
    \[
        \QC_{L_\mathbf{m}}(x)=F(x,x^{m_1},\ldots,x^{m_\ell}).
    \]
    The denominators introduced at each step are products of powers of
    $x^2+1$, $x+1$, and $x-1$, as in Lemma~\ref{mainlemma}. Hence there is a
    polynomial $D(x)$, all of whose irreducible factors are among these three
    polynomials, such that
    \[
        f(x,t_1,\ldots,t_\ell):=D(x)F(x,t_1,\ldots,t_\ell)
        \in\ZZ[x^{\pm1},t_1^{\pm1},\ldots,t_\ell^{\pm1}].
    \]
    Since $M(D)=1$, we have $M(F)=M(f)$.
    The rational-function version of Lemma~\ref{boydlawton} therefore gives
    \[
        \lim_{\nu(\mathbf{m})\to\infty}M(\QC_{L_\mathbf{m}}(x))
        =M(F)=M(f).
    \]
\end{proof}

\begin{theorem}\label{mainthtwo}
    Write $\QC_{L_m}(x) = C_0 + C_1 x^m +C_2 x^{-m}$ as in the proof of
    Theorem~\ref{mainthone}, and assume that $C_1$ and $C_2$ are not both zero.
    Let $L_m$ and $L_\mathbf{m}$ be as in Theorem~\ref{mainthone} and Corollary~\ref{maincor}.
    Suppose $L$ is an $s$-component link.
    Let $\Gamma_L(m)$ be the set of distinct nonzero roots of
    $(x+\frac{1}{x})^s\QC_{L_m}(x)$.
    Then we have
    \begin{enumerate}
        \item $\lim_{m\to\infty}\#\Gamma_L(m)=\infty$,
        \item For every $\epsilon>0$,
        \[
            \limsup_{m\to\infty}
            \#\{\gamma\in\Gamma_L(m)\mid ||\gamma|-1|>\epsilon\}<\infty.
        \]
    \end{enumerate}
    For the family $L_\mathbf{m}$, assume that for some $i$ and some fixed
    values of the other twist parameters, the resulting one-parameter family,
    obtained by varying $m_i$, has coefficients $C_1$ and $C_2$ that are not
    both zero. Let $\Gamma_L(\mathbf{m})$ be the analogous set of distinct
    nonzero roots of
    \[
        (x+x^{-1})^{s+\ell-1}\QC_{L_\mathbf{m}}(x).
    \]
    The same conclusions hold with $m\to\infty$ replaced by
    $\nu(\mathbf{m})\to\infty$.
\end{theorem}

\begin{proof}
    First, we prove (1). We define
    \begin{equation}\label{gx}
        g_m(x):=(x^2+1)^s(x+1)(x-1)^2\QC_{L_m}(x) = C_0'+C_1'x^m+C_2'x^{-m}.
    \end{equation}
    Here, $C_i'\in\ZZ[x^{\pm 1}]$ by Lemma~\ref{mainlemma}.
    There is a positive integer $M$ that bounds the sum of the absolute values
    of the coefficients of $g_m(x)$ for every $m$.
    By Lemma 1 on p.~187 of \cite{Schinzel_2000}, any polynomial with
    a root $\gamma\neq 0$ of multiplicity $n$ 
    has at least $n + 1$ nonzero coefficients.
    Thus, $n + 1 \leq M$.
    For a nonzero Laurent polynomial $h$, let
    \[
        \operatorname{br}(h)
        :=\max\operatorname{supp}(h)-\min\operatorname{supp}(h)
    \]
    denote its breadth. Since $\QC_{L_m}(x^{-1})=\QC_{L_m}(x)$, we have
    \[
        g_m(x^{-1})=x^{-(2s+3)}g_m(x).
    \]
    Thus, the support of $g_m$ is symmetric about $(2s+3)/2$.
    We also have $C_0'\neq0$. Indeed, let $\zeta=e^{\pi i/3}$. Since
    $\zeta+\zeta^{-1}=1$ and $Q_{L_m}(1)=1$ for every $m$, we have
    \[
        C_0(\zeta)+C_1(\zeta)\zeta^m+C_2(\zeta)\zeta^{-m}=1
    \]
    for every $m$. The three sequences $1,\zeta^m,\zeta^{-m}$ are linearly
    independent, and hence $C_0(\zeta)=1$. In particular, $C_0$ and $C_0'$
    are nonzero. Since $C_1'$ and $C_2'$ are not both zero, the support in
    equation~(\ref{gx}) contains, for all sufficiently large $m$, both a fixed
    exponent and an exponent whose absolute value tends to infinity. Therefore,
    \[
        \operatorname{br}(g_m)\longrightarrow\infty
        \qquad (m\to\infty).
    \]

    After multiplying $g_m$ by a suitable power of $x$, we obtain a polynomial
    of degree $\operatorname{br}(g_m)$ with a nonzero constant term. Since the
    multiplicity of each of its roots is at most $M-1$, we have
    \[
        \#\{\text{distinct nonzero roots of }g_m\}
        \geq\frac{\operatorname{br}(g_m)}{M-1}.
    \]
    Moreover, if
    \[
        P_m(x):=(x+x^{-1})^s\QC_{L_m}(x),
    \]
    then
    \[
        g_m(x)=x^s(x+1)(x-1)^2P_m(x).
    \]
    Thus the nonzero roots of $g_m$ that are not in $\Gamma_L(m)$ belong to
    the fixed set $\{-1,1\}$. Since $\operatorname{br}(g_m)\to\infty$, it
    follows that $\#\Gamma_L(m)\to\infty$, proving (1).

    Let us prove (2). It is enough to consider $0<\epsilon<1$. We set
    \begin{align*}
        A_\epsilon(m)&:=\{\alpha\in\Gamma_L(m)\mid |\alpha|\geq1+\epsilon\},\\
        B_\epsilon(m)&:=\{\beta\in\Gamma_L(m)\mid |\beta|\leq1-\epsilon\}.
    \end{align*}
    After multiplying $P_m$ by a suitable power of $x$, we may regard it as
    an integral polynomial with a nonzero constant term; this multiplication
    does not change its nonzero roots or its Mahler measure. In the following
    product, roots are counted with multiplicity, whereas the elements of
    $A_\epsilon(m)$ are counted without multiplicity. Jensen's formula gives
    \begin{equation*}
        (1+\epsilon)^{\#A_\epsilon(m)}
        \leq \prod_{|\alpha|>1}|\alpha|
        \leq M(P_m).
    \end{equation*}
    Here the last inequality follows because the absolute value of the leading
    coefficient of an integral polynomial is at least one. Furthermore,
    \[
        M(P_m)=M(\QC_{L_m}),
    \]
    since $M(x+x^{-1})=1$. Theorem~\ref{mainthone} therefore shows that
    $M(P_m)$ is bounded as $m\to\infty$.
    Hence, $\#A_\epsilon(m)$ is bounded as $m\to\infty$.
    Since $1/\gamma\in\Gamma_L(m)$ whenever $\gamma\in\Gamma_L(m)$, and
    \[
        |\gamma|\leq1-\epsilon
        \quad\Longrightarrow\quad
        |\gamma^{-1}|\geq\frac{1}{1-\epsilon}>1+\epsilon,
    \]
    $\#B_\epsilon(m)$ is also bounded as $m\to\infty$.
    Thus, the limit superior of each of $\#A_\epsilon(m)$ and
    $\#B_\epsilon(m)$ is finite, which proves (2).

    For $L_\mathbf{m}$, let $f$ be the polynomial in
    Corollary~\ref{maincor}. By the additional assumption, its support contains
    two terms with different twist-variable exponents. Since $f$ has finite
    support, the condition $\nu(\mathbf{m})\to\infty$ implies that these terms
    remain distinct after specialization and that the difference of their
    degrees tends to infinity. Hence the breadth tends to infinity, and the
    rest of the proof is the same as in the one-parameter case.
\end{proof}

\begin{remark}
    Theorem~\ref{mainthtwo} implies that, except for a uniformly bounded
    number of distinct roots, the roots of $\QC_{L_m}(x)$ approach the unit
    circle in the complex plane as $m\to\infty$. Consequently, except for a
    uniformly bounded number of distinct roots, the roots of $Q_{L_m}(x)$
    approach the interval $[-2,2]$ and hence the real axis as $m\to\infty$.
\end{remark}

\begin{lemma}\label{mainlemma}
    For an $s$-component link $L$, let $C_i$ be defined by
    \begin{equation*}
        \QC_{L_m}(x) = C_0 + C_1 x^m +C_2 x^{-m}
    \end{equation*}
    as in the proof of Theorem~\ref{mainthone}. Then
    $(x^2+1)^s(x+1)(x-1)^2C_i\in\ZZ[x^{\pm 1}]$.
\end{lemma}

\begin{proof}
    By Property 6 in \cite{MR837528}, the lowest degree of $Q_L(x)$ is $1-s$.
    A two-strand twist changes the number of components of the link by at most $1$.
    Hence each of $L_1,L_0,L_{-1}$ has at most $s+1$ components. Since
    \[
        x+x^{-1}=\frac{x^2+1}{x},
    \]
    the lowest-degree estimate implies that
    \[
        g_j(x):=(x^2+1)^s\QC_{L_j}(x)\in\ZZ[x^{\pm1}]
        \qquad (j=-1,0,1).
    \]
    Therefore, we have
    \begin{align*}
        C_0 + C_1 x +C_2 x^{-1} &= \frac{g_1(x)}{(x^2+1)^s}, \\
        C_0 + C_1 +C_2 &= \frac{g_{0}(x)}{(x^2+1)^s}, \\
        C_0 + C_1 x^{-1} +C_2 x &= \frac{g_{-1}(x)}{(x^2+1)^s},
    \end{align*}
    where $g_1(x),g_0(x),g_{-1}(x)\in\ZZ[x^{\pm 1}]$.
    Solving these equations, we obtain
    \begin{align*}
        C_0 &= \frac{(x^2+1)g_0(x) - x(g_1(x) + g_{-1}(x))}{(x^2+1)^s(x-1)^2}, \\
        C_1 &= \frac{x^2g_1(x)+xg_{-1}(x)-x(x+1)g_0(x)}{(x^2+1)^s(x+1)(x-1)^2}, \\
        C_2 &= \frac{xg_1(x)+x^2g_{-1}(x)-x(x+1)g_0(x)}{(x^2+1)^s(x+1)(x-1)^2}.
    \end{align*}
\end{proof}

\begin{example}
    Figure~\ref{fig:roottwist} shows the roots of the $Q$-polynomial and
    the Jones polynomial for twist knots.
\end{example}

\begin{figure}[htbp]
    \centering
    \begin{subfigure}[t]{0.48\textwidth}
        \centering
        \includegraphics[width=\textwidth]{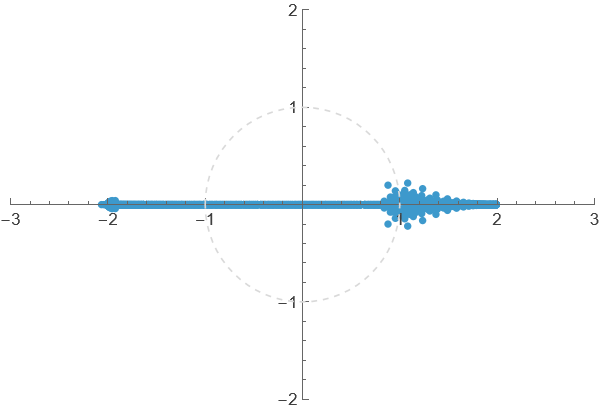}
        \caption{Roots of the $Q$-polynomial for twist knots.}
        \label{fig:Qroottwist}
    \end{subfigure}
    \hfill
    \begin{subfigure}[t]{0.48\textwidth}
        \centering
        \includegraphics[width=\textwidth]{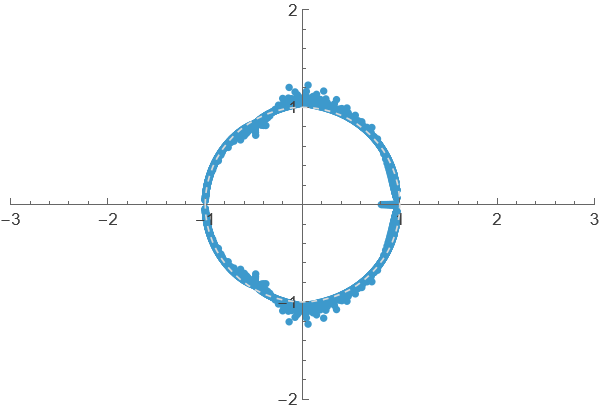}
        \caption{Roots of the Jones polynomial for twist knots.}
        \label{fig:jroottwist}
    \end{subfigure}
    \caption{Roots of the $Q$- and Jones polynomials for twist knots.}
    \label{fig:roottwist}
\end{figure}

\begin{remark}
    Champanerkar and Kofman treated full twists on an arbitrary number of
    strands for the Jones polynomial by working in the Temperley--Lieb algebra
    $TL_n$ over $\QQ(A)$~\cite{10.2140/agt.2005.5.1}. Its semisimplicity allows
    the full twist to be decomposed using mutually orthogonal minimal central
    idempotents, so that its powers can be analyzed explicitly. A natural analogue
    for the Kauffman polynomial would use the Birman--Murakami--Wenzl algebra
    \cite{Murakami1987,BirmanWenzl1989}. Although the BMW algebra is semisimple for generic
    parameters, the specialization $a=1$ defining the $Q$-polynomial corresponds
    to the exceptional parameter $r=1$ in the usual BMW notation; for
    $n\geq3$, the resulting algebra is not semisimple in general
    \cite{RuiSi2009}. Consequently, the decomposition into minimal central
    idempotents used in the Temperley--Lieb argument cannot simply be
    specialized to the $Q$-polynomial. This is why we use the direct recurrence
    for two-strand twists and do not claim an analogous result for full twists
    on an arbitrary number of strands.
\end{remark}

\section{Stability of knot polynomials}
\label{sec:comparison}
In this section, we discuss the roots of the $Q$-polynomial and compare them
with those of other knot polynomials.

\subsection{Examples of r-stable $Q$-polynomials}
Hirasawa and Murasugi have studied stable Alexander polynomials; see, for
example,~\cite{doi:10.1142/S0218216519400170}. We discuss two notions of
stability for the Alexander, Jones, and $Q$-polynomials.
\begin{definition}[\cite{doi:10.1142/S0218216519400170}]
    A polynomial is said to be real stable (or, briefly, r-stable)
    if all its zeros are real.
    A polynomial is said to be circular stable (or, briefly, c-stable)
    if all its zeros have modulus 1.
    For a Laurent polynomial, these terms refer to its nonzero zeros.
\end{definition}

Our numerical computations suggest that the $Q$-polynomials of the knots
listed in Table~\ref{fig:Qrealroottable} are r-stable.
The knot table used in these computations was taken from the Mathematica
package \texttt{KnotTheory\textasciigrave}~\cite{KnotTheoryPackage}, and the
roots were computed numerically in Mathematica.

\begin{table}[htbp]
  \centering
  \begin{tabular}{lccccccccccccccc} 
    $3_1$ \\
    $4_1$ \\
    $5_2$ \\
    $6_3$ \\
    $7_7$ \\
    $8_{16}$ \\
    $9_{31}$ & $9_{42}$ & $9_{43}$ \\
    $10_{45}$ & $10_{161}$ \\
    $11_{n20}$ & $11_{n96}$ & $11_{n111}$ & $11_{n128}$ & $11_{n133}$ & $11_{n135}$ & $11_{n147}$ & $11_{n149}$ & $11_{n183}$ \\
  \end{tabular}
  \caption{Knots whose $Q$-polynomials appear numerically to be r-stable.}
  \label{fig:Qrealroottable}
\end{table}

Inspection of Table~\ref{fig:Qrealroottable} led us to the following infinite
family of $2$-bridge links.

\begin{proposition}\label{realpro}
    For every positive integer $n$, the $Q$-polynomial of the $2$-bridge link
    $K_n:=[1,1,\ldots,1,1]$ is r-stable, where the continued fraction has $2n$
    entries, all equal to $1$.
\end{proposition}

\begin{proof}
    The $Q$-polynomials of the links $K_n$ satisfy the recurrence relation
    \begin{equation*}
        Q_{K_n}(x) - (x^2+2x)Q_{K_{n-1}}(x) + (x^2+2x)Q_{K_{n-2}}(x) -Q_{K_{n-3}}(x) = 0,
    \end{equation*}
    with initial conditions $ Q_{K_{-1}}(x) = Q_{K_{0}}(x)=1 $ and $Q_{K_1}(x)=2x+1-2x^{-1}$.
    Let $c_n := Q_{K_n}(x) - Q_{K_{n-1}}(x)$. Then
    \begin{equation*}
        c_n(x) = (x^2+2x-1)c_{n-1}(x) - c_{n-2}(x)
    \end{equation*}
    holds with initial conditions $c_0(x)=0,c_1(x) = 2(x-x^{-1})$.

    The key observation is that $c_n(x)=2(x-x^{-1})U_n(\frac{x^2+2x-1}{2})$,
    where we use the shifted indexing
    \[
        U_n(\cos\theta)=\frac{\sin n\theta}{\sin\theta}.
    \]
    Thus, for $n\geq1$, our $U_n$ is the polynomial usually denoted by
    $U_{n-1}$. We then have
    \begin{equation*}
        Q_{K_n}(x) = 1 + 2(x-x^{-1})\sum_{k=0}^{n}U_k(\frac{x^2+2x-1}{2}).
    \end{equation*}
    Let $x := 2\cos\theta-1$. Using the formula $2\sin\alpha\sin\beta=\cos(\alpha-\beta)-\cos(\alpha+\beta)$,
    we have
    \begin{equation*}
        2\sin\theta\sum_{k=0}^{n}\frac{\sin 2k\theta}{\sin2\theta} = \frac{\cos(\theta)-\cos((2n+1)\theta)}{\sin2\theta}.
    \end{equation*}
    Therefore,
    \begin{align}
        Q_{K_n}(x)&=1 + \frac{8\cos\theta(\cos\theta-1)}{2\cos\theta-1}\cdot\frac{\cos(\theta)-\cos((2n+1)\theta)}{2\sin\theta\sin2\theta} \\
        &=1-\frac{2(\cos(\theta)-\cos((2n+1)\theta))}{(2\cos\theta-1)(\cos\theta+1)}\\
        &=\frac{-\cos(\theta)+\cos2\theta+2\cos((2n+1)\theta)}{(2\cos\theta-1)(\cos\theta+1)}.\label{nahan}
    \end{align}

    Let the numerator in equation~(\ref{nahan}) be
    \[
        g(\theta)=-\cos\theta+\cos2\theta+2\cos((2n+1)\theta).
    \]
    Set $\theta_k=k\pi/(2n+1)$ for $0\leq k\leq2n+1$. Since
    \[
        g(\theta_k)
        =2(-1)^k+2\cos^2\theta_k-\cos\theta_k-1,
    \]
    and
    \[
        -\frac98\leq 2u^2-u-1\leq2
        \qquad(-1\leq u\leq1),
    \]
    we have $g(\theta_k)>0$ when $k$ is even and
    $g(\theta_k)<0$ when $k<2n+1$ is odd. Moreover, $g(\pi)=0$.
    Hence $g$ has a zero in each interval
    $(\theta_k,\theta_{k+1})$ for $0\leq k\leq2n-1$, together with the zero
    $\theta=\pi$. These are $2n+1$ distinct zeros. Since $g(\theta)$ is a
    polynomial of degree $2n+1$ in $u=\cos\theta$, and $\cos\theta$ is injective
    on $[0,\pi]$, these are all the zeros of $g$, and they are simple when $g$
    is regarded as a polynomial in $u$.
    Since $x=2u-1$, the numerator of~(\ref{nahan}) has degree $2n+1$ in
    $x$, whereas its denominator has degree $2$. After the cancellations
    described below, it follows that the highest degree of $Q_{K_n}(x)$ is
    $2n-1$.

    The details differ according to whether $K_n$ is a knot or a link
    (see Remark~\ref{remlink}).
    \begin{enumerate}
        \item Suppose that $n\equiv0,2\pmod3$, so that $K_n$ is a knot.
        In this case $2n+1\equiv1,5\pmod6$, and hence
        \[
            g\left(\frac{\pi}{3}\right)
            =-1+2\cos\left(\frac{(2n+1)\pi}{3}\right)=0.
        \]
        Thus the zeros at $\theta=\pi/3$ and $\theta=\pi$ cancel the two
        factors in the denominator of~(\ref{nahan}). The lowest degree of
        $Q_{K_n}(x)$ is $0$. The remaining
        $2n-1$ zeros are distinct and real, so $Q_{K_n}(x)$ is r-stable.

        \item Suppose that $n\equiv1\pmod3$, so that $K_n$ is a
        $2$-component link. In this case $2n+1\equiv3\pmod6$, and hence
        \[
            g\left(\frac{\pi}{3}\right)=-3\neq0.
        \]
        The zero at $\theta=\pi$ cancels the factor $\cos\theta+1$, whereas
        $\theta=\pi/3$, corresponding to $x=0$, gives a pole. The lowest
        degree of $Q_{K_n}(x)$ is $-1$.
        Therefore $xQ_{K_n}(x)$ has the remaining $2n$ distinct real
        zeros. Equivalently, all nonzero zeros of $Q_{K_n}(x)$ are real, so
        $Q_{K_n}(x)$ is r-stable.
    \end{enumerate}
\end{proof}

\begin{figure}
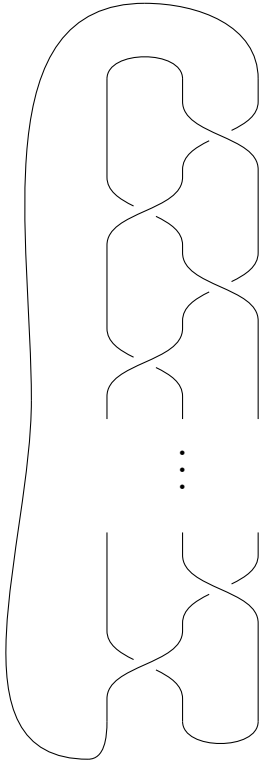

    \twobridge
    \caption{$2$-bridge links $[1,1,\ldots,1,1]$}
    \label{fig:twobridgeknots}
\end{figure}

\begin{remark}\label{remlink}
    $K_n$ is a knot if $n\equiv0,2$ modulo $3$.
    It is a $2$-component link if $n\equiv1$ modulo $3$.
\end{remark}

\begin{remark}
    We have verified the analogue of Proposition~\ref{realpro}, with $2n+1$
    entries equal to $1$, for a few small values of $n$, but we do not yet
    have a complete proof.
\end{remark}

\subsection{r-stability and c-stability of knot polynomials}

It is known that the Alexander polynomials of torus knots are c-stable.
There are also twist knots whose Alexander polynomials are r-stable.

Table~\ref{tab:stability-comparison} summarizes the contrasting behavior of
the three knot polynomials. Here ``Yes'' means that a nonconstant example
exists, whereas ``No'' means that no nonconstant example exists.

\begin{table}[htbp]
    \centering
    \begin{tabular}{c|cc}
        \textbf{Polynomial} & \textbf{c-stable examples} &
        \textbf{r-stable examples} \\
        \hline
        Alexander & Yes & Yes \\
        Jones     & Yes~\cite{Mroczkowski2022} &
        No (Proposition~\ref{prop:jones-not-r-stable}) \\
        $Q$       & No (Proposition~\ref{prop:q-not-c-stable}) &
        Yes (Proposition~\ref{realpro})
    \end{tabular}
    \caption{Existence of nonconstant c-stable and r-stable knot polynomials.}
    \label{tab:stability-comparison}
\end{table}

For the Jones polynomial, we first record the following proposition.

\begin{proposition}\label{prop:jones-not-r-stable}
    The Jones polynomial of a knot cannot be both nonconstant and r-stable.
\end{proposition}

\begin{proof}
    Suppose, to the contrary, that $J_K(q)$ is nonconstant and r-stable.
    First, $J_K(q)$ cannot be a Laurent monomial. Indeed, $J_K(1)=1$ and
    $J_K'(1)=0$ by \cite[Section~12.2]{Jones1987HeckeAR}; if
    $J_K(q)=a q^r$, these identities imply $a=1$ and $r=0$, contrary to the
    assumption that $J_K(q)$ is nonconstant. Thus $J_K(q)$ has a nonzero root.
    Then we have the factorization
    \[
        J_K(q)=a q^r\prod_{i=1}^d(q-\alpha_i),
    \]
    where $d>0$, $\alpha_i\in\mathbb{R}\setminus\{0\}$, and $a$ is the leading coefficient of $J_K(q)$.
    By Section~12 of \cite{Jones1987HeckeAR},
    we have $|J_K(q)|^2 = 1$ at $q = 1, i, e^{\frac{2\pi i}{3}}$.
    \begin{align*}
        f(\theta) &:= |J_K(e^{i \theta})|^2 = |a|^2 \prod_{i=1}^d (e^{i \theta} - \alpha_i)(e^{-i \theta} - \alpha_i) \\
        &=|a|^2 \prod_{i=1}^d (1 - 2\alpha_i \cos(\theta) + \alpha_i^2).
    \end{align*}
    Put $u=\cos\theta$ and define
    \[
        F(u):=\log\left(
        |a|^2\prod_{i=1}^d(1-2\alpha_i u+\alpha_i^2)
        \right)
        \qquad (-1/2\leq u\leq1).
    \]
    Then
    \[
        F''(u)
        =-\sum_{i=1}^d
        \frac{4\alpha_i^2}{(1-2\alpha_i u+\alpha_i^2)^2}<0,
    \]
    so $F$ is strictly concave on $[-1/2,1]$. On the other hand, the three
    special values above give
    \[
        F(1)=F(0)=F(-1/2)=0.
    \]
    Strict concavity applied to the endpoints $-1/2$ and $1$ would instead
    give $F(0)>0$, a contradiction. Therefore, no nonconstant Jones polynomial
    of a knot is r-stable.
\end{proof}

Mroczkowski~\cite{Mroczkowski2022} found an infinite family of knots whose
Jones polynomials are c-stable.

In contrast, we have the following proposition for the $Q$-polynomial.

\begin{proposition}\label{prop:q-not-c-stable}
    The $Q$-polynomial of a knot cannot be both nonconstant and c-stable.
\end{proposition}

\begin{proof}
    For any knot $K$, we have $Q_K(1) = 1$ and $Q_K(-2) = 1$.
    Suppose that $Q_K(x)$ is nonconstant and c-stable.
    Write
    \[
        Q_K(x)=a\prod_{j=1}^d(x-\alpha_j),
    \]
    where $d>0$, $\alpha_j=e^{i\theta_j}$, and $a$ is the leading coefficient
    of $Q_K(x)$. We then have
    \begin{align*}
        |Q_K(-2)|^2 &= |a|^2\prod_{j=1}^d |(-2) - \alpha_j|^2 \\
        &= |a|^2\prod_{j=1}^d (4 + 4\cos\theta_j + 1) \\
        &= |a|^2\prod_{j=1}^d (5 + 4\cos\theta_j) \\
        &\geq |a|^2\geq 1,
    \end{align*}
    Equality is possible only if $|a|=1$ and $\alpha_j=-1$ for every $j$.
    In this case, $Q_K(1)=\pm2^d\neq1$. This is a contradiction. Therefore,
    no nonconstant $Q$-polynomial of a knot is c-stable.
\end{proof}

\section*{Acknowledgements}
The author thanks Yeonhee Jang for her thoughtful guidance and helpful discussions,
Taizo Kanenobu for pointing out the $2$-bridge description $[1,1,\ldots,1,1]$ of one
of the examples, an observation that led to the family studied in Proposition~\ref{realpro},
Jumpei Yasuda and Mikami Hirasawa for helpful discussions about this work.
The author used GPT-5.6 Sol and GPT-6 Astra via Codex while preparing this article and takes full
responsibility for its contents.
This work was supported by JST SPRING, Grant Number JPMJSP2139.

\bibliographystyle{plain}
\bibliography{sanko}

@article{10.2140/agt.2005.5.1,
author = {Abhijit Champanerkar and Ilya Kofman},
title = {{On the Mahler measure of Jones polynomials under twisting}},
volume = {5},
journal = {Algebraic \& Geometric Topology},
number = {1},
publisher = {MSP},
pages = {1 -- 22},
year = {2005},
doi = {10.2140/agt.2005.5.1},
URL = {https://doi.org/10.2140/agt.2005.5.1}
}

@article{BirmanWenzl1989,
  author = {Joan S. Birman and Hans Wenzl},
  title = {Braids, link polynomials and a new algebra},
  journal = {Transactions of the American Mathematical Society},
  volume = {313},
  number = {1},
  pages = {249--273},
  year = {1989},
  doi = {10.1090/S0002-9947-1989-0992598-X},
  url = {https://doi.org/10.1090/S0002-9947-1989-0992598-X}
}

@article{Murakami1987,
  author = {Jun Murakami},
  title = {The {Kauffman} polynomial of links and representation theory},
  journal = {Osaka Journal of Mathematics},
  volume = {24},
  number = {4},
  pages = {745--758},
  year = {1987},
  doi = {10.18910/4549},
  url = {https://doi.org/10.18910/4549}
}

@article{RuiSi2009,
  author = {Hebing Rui and Mei Si},
  title = {Gram determinants and semisimplicity criteria for {Birman--Wenzl}
           algebras},
  journal = {Journal f{\"u}r die reine und angewandte Mathematik},
  volume = {631},
  pages = {153--179},
  year = {2009},
  doi = {10.1515/CRELLE.2009.045},
  url = {https://doi.org/10.1515/CRELLE.2009.045}
}

@article {MR837528,
    AUTHOR = {Brandt, Robert D. and Lickorish, W. B. R. and Millett, Kenneth
              C.},
     TITLE = {A polynomial invariant for unoriented knots and links},
   JOURNAL = {Invent. Math.},
  FJOURNAL = {Inventiones Mathematicae},
    VOLUME = {84},
      YEAR = {1986},
    NUMBER = {3},
     PAGES = {563--573},
      ISSN = {0020-9910,1432-1297},
   MRCLASS = {57M25},
  MRNUMBER = {837528},
MRREVIEWER = {Cameron\ McA.\ Gordon},
       DOI = {10.1007/BF01388747},
       URL = {https://doi.org/10.1007/BF01388747},
}

@article{Mroczkowski2022,
  author = {Maciej Mroczkowski},
  title = {Infinitely many roots of unity are zeros of some {Jones} polynomials},
  journal = {Geometriae Dedicata},
  volume = {216},
  number = {4},
  year = {2022},
  month = jun,
  pages = {43},
  doi = {10.1007/s10711-022-00708-4},
  url = {https://doi.org/10.1007/s10711-022-00708-4}
}

@article{Jones1987HeckeAR,
  title={{Hecke} algebra representations of braid groups and link polynomials},
  author={Vaughan F. R. Jones},
  journal={Annals of Mathematics},
  year={1987},
  volume={126},
  pages={335-388},
  url={https://api.semanticscholar.org/CorpusID:122120158}
}

@article{Ho1985,
  author  = {Ho, C. F.},
  title   = {A polynomial invariant for knots and links---preliminary report},
  journal = {Abstracts of the American Mathematical Society},
  volume  = {6},
  year    = {1985},
  pages   = {300}
}

@article{LAWTON1983356,
title = {A problem of {Boyd} concerning geometric means of polynomials},
journal = {Journal of Number Theory},
volume = {16},
number = {3},
pages = {356-362},
year = {1983},
issn = {0022-314X},
doi = {https://doi.org/10.1016/0022-314X(83)90063-X},
url = {https://www.sciencedirect.com/science/article/pii/0022314X8390063X},
author = {Wayne M. Lawton}
}

@article{10.1007/BF02417878,
author = {J. L. W. V. Jensen},
title = {Sur un nouvel et important th{\'e}or{\`e}me de la th{\'e}orie des fonctions},
volume = {22},
journal = {Acta Mathematica},
publisher = {Institut Mittag-Leffler},
pages = {359 -- 364},
year = {1900},
doi = {10.1007/BF02417878},
URL = {https://doi.org/10.1007/BF02417878}
}

@book{Schinzel_2000, 
place={Cambridge}, 
series={Encyclopedia of Mathematics and its Applications}, 
title={Polynomials with Special Regard to Reducibility}, 
publisher={Cambridge University Press}, 
author={Schinzel, A.}, 
year={2000}, 
collection={Encyclopedia of Mathematics and its Applications}
}

@article{doi:10.1142/S0218216519400170,
author = {Hirasawa, Mikami and Murasugi, Kunio},
title = {Stable {Alexander} polynomials of arborescent links},
journal = {Journal of Knot Theory and Its Ramifications},
volume = {28},
number = {13},
pages = {1940017},
year = {2019},
doi = {10.1142/S0218216519400170},
URL = { 
        https://doi.org/10.1142/S0218216519400170
},
eprint = { 
        https://doi.org/10.1142/S0218216519400170
}
}

@article {MR2050045,
    AUTHOR = {Silver, Daniel S. and Williams, Susan G.},
     TITLE = {Mahler measure of {A}lexander polynomials},
   JOURNAL = {J. London Math. Soc. (2)},
  FJOURNAL = {Journal of the London Mathematical Society. Second Series},
    VOLUME = {69},
      YEAR = {2004},
    NUMBER = {3},
     PAGES = {767--782},
      ISSN = {0024-6107,1469-7750},
   MRCLASS = {57M25 (11R06)},
  MRNUMBER = {2050045},
       DOI = {10.1112/S0024610704005289},
       URL = {https://doi.org/10.1112/S0024610704005289},
}

@misc{KnotTheoryPackage,
  author = {Bar-Natan, Dror},
  title = {The {Mathematica} package {KnotTheory}},
  year = {2005},
  howpublished = {Knot Atlas},
  url = {https://katlas.org/wiki/The_Mathematica_Package_KnotTheory},
  note = {Accessed September 4, 2026}
}

\end{document}